\documentclass[11pt,centertags,reqno,twoside]{amsart}
\usepackage{amsmath,latexsym, graphicx}
\usepackage[psamsfonts]{amssymb}
\usepackage{mathptm}
\usepackage{color}
\usepackage{amsmath}
\usepackage{amsfonts}
\usepackage{amssymb}
\usepackage{amsthm}
\usepackage{lineno}
\usepackage[mathcal]{euscript}
\usepackage{bm}
\usepackage[english]{babel}
\usepackage[bookmarks]{hyperref}
\usepackage{ulem}
\usepackage{caption}
\usepackage{subcaption}
\numberwithin{equation}{section}

\renewcommand{\epsilon}{\varepsilon}

\renewcommand{\epsilon}{\varepsilon}

\renewcommand{\hat}{\widehat }

\newcommand{\black}{\color{black}}

\newcommand{\be}{\begin{equation}}
\newcommand{\ee}{\end{equation}}

\renewcommand{\det}{\mathop{\mathrm{det}}}

{\bf}{\it}
\newtheorem{theorem}{Theorem}[section]
\newtheorem{lemma}[theorem]{Lemma}
\newtheorem{corollary}[theorem]{Corollary}

\newtheorem{definition}[theorem]{Definition}
\newtheorem{proposition}[theorem]{Proposition}

\newtheorem{remark}[theorem]{Remark}

\date{\today}

\begin{document}

\author{Saidakhmat N. Lakaev, Saidakbar S.~Abduvayitov, Shuhrat S.~Lakaev}

\address[Saidakhmat Lakaev]{Samarkand State University, 140104, Samarkand, Uzbekistan}
\email{slakaev2019@gmail.com}

\address[Saidakbar Abduvayitov]{Samarkand State University, 140104, Samarkand, Uzbekistan}
\email{saidakbar1911@gmail.com}

\address[Shuhrat Lakaev]{National University of Uzbekistan named after Mirzo Ulugbek, 100123, Tashkent, Uzbekistan}
\email{shlakaev@mail.ru}

\large
\title[Threshold Resonances, Critical Couplings, and Eigenvalue Bounds ...]{Threshold Resonances, Critical Couplings, and Eigenvalue Bounds for Two-Particle Operators on $\mathbb{Z}^3$}

\maketitle

\begin{abstract}
We study a family of lattice Schr\"odinger operators $H_{\mu_1\mu_2}(K)$ describing two identical bosons on the three-dimensional cubic lattice $\mathbb{Z}^3$, where $K \in \mathbb{T}^3$ is the quasi-momentum, and $\mu_1, \mu_2 \in \mathbb{R}$ are coupling constants corresponding to on-site and nearest-neighbour interactions, respectively.

We show that the Hilbert space $L^{2,\mathrm{e}}(\mathbb{T}^3)$ decomposes into three mutually orthogonal subspaces, each invariant under $H_{\mu_1\mu_2}(0)$. A detailed spectral analysis of the restriction of $H_{\mu_1\mu_2}(0)$ to one of these subspaces reveals two smooth critical curves in the $(\mu_1, \mu_2)$-plane, separating regions where the number of eigenvalues below the essential spectrum remains constant. For the restrictions to the other two subspaces, we identify a critical point on the $\mu_2$-axis that partitions it into intervals with a constant number of eigenvalues below the essential spectrum.

Analogously, two additional critical curves and one critical point determine regions and intervals where the number of eigenvalues above the essential spectrum is constant. In particular, for suitable parameter values, $H_{\mu_1\mu_2}(0)$ may possess up to three bound states located either below or above the essential spectrum, with numbers $(\alpha,\beta)$ satisfying $\alpha + \beta < 3$ or $(\alpha,\beta) \in \{(3,0),(0,3)\}$.

Here, by \emph{eigenvalue bounds} we mean both the possible locations of eigenvalues outside the essential spectrum and the maximum number of such eigenvalues for given coupling parameters. Finally, we extend the analysis to arbitrary quasi-momentum $K \in \mathbb{T}^3$, obtaining general lower bounds for the number of eigenvalues of $H_{\mu_1\mu_2}(K)$.
\end{abstract}



\section{Introduction}
Lattice models constitute an important class of operators in mathematical physics, offering a rigorous framework for analyzing quantum systems evolving in discrete media. In this context, few-body Hamiltonians \cite{Mattis:1986} serve as the simplest version of the Bose--Hubbard model  and describe the motion of a finite number of identical particles subject to short-range interactions. Over the past several decades, such operators have been widely studied \cite{ALzM:2004,ALMM:2006,ALzM:2007,BdsPL:2017,FICarroll:2002,HMumK:2020,LO'zdemir:2016,LKhKh:2021,Motovilov:2001,KhLAlmuratov:2022} both for their relevance in condensed-matter physics and for their role as mathematically accessible discretizations of continuous Schr\"odinger operators.

A major motivation for the study of lattice Hamiltonians stems from their interpretation as discrete approximations of continuous few-body Schr\"odinger operators \cite{Faddeev:1986}. Placing the $N$-body problem on a lattice yields a setting where all operators are bounded, simplifying many functional-analytic questions. For example, the one-particle lattice Hamiltonian in one dimension is essentially governed by the classical perturbation theory of infinite Jacobi matrices (see, e.g., \cite{Teschl2000,Yafaev2017}).

Lattice Schr\"{o}dinger operators also provide realistic mathematical models for systems of particles propagating in periodic structures, such as ultracold atoms in optical crystals \cite{Bloch:2005,Winkler:2006}. Advances in optical-lattice experiments have enabled unprecedented control of physical parameters---temperature, atomic species, interaction strengths, and trapping potentials---leading to a renewed interest in few-body phenomena in these systems (see, e.g., \cite{Bloch:2005,JBC:1998,JZ:2005,Lewenstein:2012}).

A further source of motivation comes from the analysis of \emph{threshold phenomena}, which play a central role in modern few-body spectral theory. Remarkably, discrete lattice models often reproduce the key analytical features of their continuous counterparts, and this parallel has repeatedly proved to be both robust and insightful. A prominent example is the \emph{Efimov effect} \cite{Efimov:1970}, first predicted for three-body systems in $\mathbb{R}^3$ and later established rigorously in a series of mathematical works \cite{Ovchinnikov:1979,Sobolev:1993,Tamura:1991,Yafaev:1974}. 

Discrete analogues of the Efimov effect have been established for systems of three indistinguishable particles on a $\mathbb{Z}^3$ lattice. However, these results have so far been confined to zero-range, on-site pair interactions, as shown initially in \cite{Lakaev:1993} and later generalized for arbitrary particles in \cite{ALzM:2004}. This demonstrated universality across continuous and discrete settings now motivates a systematic study of lattice Hamiltonians with finite-range, compactly supported interactions. 

In the present paper, we study a two-parameter family of two-particle lattice 
Schr\"odinger operators
\begin{displaymath}
H_{\mu_1\mu_2}(K)=H_0(K)+V_{\mu_1\mu_2},
\end{displaymath}
acting in the Hilbert space $L^{2,\mathrm e}(\mathbb{T}^3)$, where 
$K\in\mathbb{T}^3$ denotes the total quasi-momentum. 
Here, $H_0(K)$ is the free Hamiltonian on the cubic lattice $\mathbb Z^3$, 
and the interaction term $V_{\mu_1\mu_2}$ depends on two real coupling 
parameters: the parameter $\mu_1$ represents the on-site (contact) 
interaction, while $\mu_2$ accounts for nearest-neighbor interactions 
between particles located at adjacent lattice sites.

In total quasi-momentum $K=0$, the operator $H_{\mu_1\mu_2}(0)$ leaves invariant 
the subspace of even functions $L^{2,\mathrm e}(\mathbb{T}^3)$. 
We demonstrate that this space admits a canonical orthogonal decomposition into 
three invariant subspaces,
\begin{displaymath}
L^{2,\mathrm e}(\mathbb{T}^3)
=
L^{2,\mathrm{e,s}}(\mathbb{T}^3)
\oplus
L^{2,\mathrm{e,a_{12}}}(\mathbb{T}^3)
\oplus
L^{2,\mathrm{e,mix}}(\mathbb{T}^3),
\end{displaymath}
corresponding respectively to the totally symmetric sector, the antisymmetric sector
 with respect to the exchange of the first two momentum components, 
and a mixed symmetry sector. 
Each of these subspaces is invariant under the action of $H_{\mu_1\mu_2}(0)$, 
and this decomposition reduces the spectral analysis to that of three 
simplified operators.

The main goal of this paper is to provide an explicit description of the critical finite-rank perturbations $V_{\mu_1\mu_2}$ for which small variations can cause eigenvalues of the operator $H_0(0) + V_{\mu_1\mu_2}$ to emerge or disappear outside its essential spectrum. The threshold analysis of the perturbation determinant, based on a Rouch\'e-type argument, shows that $H_{\mu_1\mu_2}(0)$ is critical at the edge of the spectrum if and only if the potential parameters satisfy a specific algebraic condition.

In particular, the analysis of the restrictions of $H_{\mu_1\mu_2}(0)$ to the invariant subspaces reveals a rich geometric structure in the parameter space $(\mu_1, \mu_2)$. In the symmetric subspace, two smooth curves partition the plane into regions where the number of eigenvalues below the bottom of the essential spectrum remains constant. In the other two subspaces, spectral transitions are governed by a single critical point on the $\mu_2$-axis. A complementary study above the top of the essential spectrum identifies two additional critical curves and one further critical point, which together determine the number of eigenvalues above the essential spectrum.

As a consequence, the operator $H_{\mu_1\mu_2}(0)$ may possess up to three eigenvalues lying outside its essential spectrum. Denoting by $\alpha$ and $\beta$ the numbers of eigenvalues below and above the essential spectrum, respectively, we prove that
$$
\alpha+\beta < 3\quad\text{or}\quad (\alpha,\beta)=\{(3,0), (0,3)\}.
$$
Notably, and in contrast with continuous Schr\"odinger operators, the lattice operator $H_{\mu_1\mu_2}(0)$ may exhibit eigenvalues simultaneously below \emph{and} above the essential spectrum.

Finally, we establish a sharp lower bound for the number of eigenvalues of \\
$H_{\mu_1\mu_2}(K)$, uniform in $K\in\mathbb{T}^3$, expressed explicitly in terms of the coupling parameters.

The study of lattice Schr\"{o}dinger operators with on-site and nearest-neighbor interactions was initiated in \cite{LBozorov:2009} for non-negative coupling constants. Subsequent works extended the analysis to one- and two-dimensional bosonic systems with arbitrary real parameters \cite{LO'zdemir:2016,LKhKh:2021}. More recently, models incorporating next-nearest-neighbour interactions were analyzed in one and two dimensions \cite{LACAOT:2023,LMA:2023}, where the precise number and distribution of eigenvalues were determined.

Thus the paper presents a detailed spectral analysis of a two-particle Schr\"odinger operators on the three-dimensional lattice $\mathbb{Z}^3$. The interactions are supported on a single site and its nearest-neighbor lattice points. These interaction structures play an important role in the study of three-body phenomena. Our focus is on characterizing the discrete spectrum and investigating the potential emergence of Efimov-like effects in three-particle systems with such finite-range interactions.

The paper is structured as follows.
Section \ref{sec:Hamiltonian} introduces the two-particle Hamiltonian in both position and quasi-momentum representations,  lattice Schr\"odinger operators  and describes its essential spectrum. Section \ref{sec:InvariantSubspaces} establishes the symmetry-induced invariant subspace decomposition and defines the corresponding restricted operators.
Section \ref{sec:MainResults} summarizes the main spectral results, including parameter–space classifications and the auxiliary lemmas and propositions required for their proofs.
Finally, Section \ref{sec:Proofs} contains the detailed proofs of the theorems stated in Section \ref{sec:MainResults}.

\section{Hamiltonian of a two-boson system on the three-dimensional lattice $\mathbb{Z}^3$
and its basic properties} 
\label{sec:Hamiltonian}

\subsection{Position-space representation}

Let $\mathbb{Z}^3 := \mathbb{Z}\times\mathbb{Z}\times\mathbb{Z}$ be the three-dimensional cubic lattice, and let
$\ell^{2,\mathrm{s}}(\mathbb{Z}^3\times\mathbb{Z}^3)$ denote the Hilbert space of square-summable
functions on $\mathbb{Z}^3\times\mathbb{Z}^3$ that are symmetric with respect to the exchange of
variables.  
We consider two identical spinless bosons moving on $\mathbb{Z}^3$ and interacting
via an on-site and a nearest-neighbour pair potential.  
The interaction strengths are denoted by $\mu_1,\mu_2\in\mathbb{R}$.

The two-particle Hamiltonian is defined by
\begin{displaymath}
\widehat{\mathbb{H}}_{\mu_1\mu_2}
= \widehat{\mathbb{H}}_0 + \widehat{\mathbb{V}}_{\mu_1\mu_2},
\qquad
\widehat{\mathbb{H}}_{\mu_1\mu_2}:
\ell^{2,\mathrm{s}}(\mathbb{Z}^3\times\mathbb{Z}^3) \to 
\ell^{2,\mathrm{s}}(\mathbb{Z}^3\times\mathbb{Z}^3).
\end{displaymath}

\paragraph{The free Hamiltonian}
The operator $\widehat{\mathbb{H}}_0$ acts as
\begin{displaymath}
(\widehat{\mathbb{H}}_0 \hat f)(x_1,x_2)
= \sum_{y\in\mathbb{Z}^3} \hat\varepsilon(x_1-y)\hat f(y,x_2)
+ \sum_{y\in\mathbb{Z}^3} \hat\varepsilon(x_2-y)\hat f(x_1,y),
\end{displaymath}
where the one-particle hopping kernel is given by
\begin{displaymath}
\hat\varepsilon(s)=
\begin{cases}
6,  & |s|=0,\\
-1, & |s|=1,\\
0,  & |s|>1,
\end{cases}
\qquad |s|:=|s_1|+|s_2|+|s_3|.
\end{displaymath}
Thus $\widehat{\mathbb{H}}_0$ corresponds to the discrete Laplacian on $\mathbb{Z}^3$
shifted by $6$.

\paragraph{The interaction}
The interaction acts as multiplication by a function depending only on the
relative coordinate:
\begin{displaymath}
(\widehat{\mathbb{V}}_{\mu_1\mu_2}\hat f)(x_1,x_2)
= \widehat{v}_{\mu_1\mu_2}(x_1-x_2)\hat f(x_1,x_2),
\end{displaymath}
where
\begin{displaymath}
\widehat{v}_{\mu_1\mu_2}(x)=
\begin{cases}
\mu_1, & |x|=0,\\
\mu_2, & |x|=1,\\
0,     & |x|>1.
\end{cases}
\end{displaymath}

\subsection{Quasimomentum representation}

Let $\mathbb{T}^3 := (\mathbb{R}/2\pi\mathbb{Z})^3$ and denote by $L^2(\mathbb{T}^3)$ the space of
square-integrable functions on the torus.  
The Fourier transform $\mathcal{F}:\ell^2(\mathbb{Z}^3)\to L^2(\mathbb{T}^3)$ is defined by
\begin{displaymath}
(\mathcal{F}\hat f)(p)
= \frac{1}{(2\pi)^{3/2}}\sum_{x\in\mathbb{Z}^3}\hat f(x)e^{ip\cdot x},
\qquad
p\in\mathbb{T}^3,
\end{displaymath}
,where $p\cdot x=p_1\cdot x_1+p_2\cdot x_2+p_3\cdot x_3$.
Its inverse is given by
\begin{displaymath}
(\mathcal{F}^{-1}f)(x)
= \frac{1}{(2\pi)^{3/2}}\int_{\mathbb{T}^3} f(p)e^{-ip\cdot x}\,\mathrm{d}p.
\end{displaymath}

Under the unitary map $\mathcal{F}\otimes\mathcal{F}$, the Hamiltonian becomes
\begin{displaymath}
\mathbb{H}_{\mu_1\mu_2}
= (\mathcal{F}\otimes\mathcal{F})\,\widehat{\mathbb{H}}_{\mu_1\mu_2}\,
  (\mathcal{F}\otimes\mathcal{F})^{-1}
= \mathbb{H}_0 + \mathbb{V}_{\mu_1\mu_2}.
\end{displaymath}

The free part is the multiplication operator
\begin{displaymath}
(\mathbb{H}_0 f)(p,q)
= \bigl[\varepsilon(p)+\varepsilon(q)\bigr]\,f(p,q), 
\end{displaymath}
with single-particle dispersion relation
\begin{displaymath}
\varepsilon(p)
= 2\sum_{i=1}^3 \bigl(1 - \cos p_i\bigr), \qquad p\in\mathbb{T}^3.
\end{displaymath}

The interaction is given by
\begin{displaymath}
(\mathbb{V}_{\mu_1\mu_2}f)(p,q)
= \frac{1}{(2\pi)^3}\int_{\mathbb{T}^3}
  v_{\mu_1\mu_2}(p-u)\,
  f(u,p+q-u)\,\mathrm{d}u,
\end{displaymath}
where $v_{\mu_1\mu_2}$ is the Fourier transform of $\widehat{v}_{\mu_1\mu_2}$:
\begin{displaymath}
v_{\mu_1\mu_2}(p)
= \mu_1 + 2\mu_2\sum_{i=1}^3 \cos p_i.
\end{displaymath}

\subsection{Decomposition with respect to the total quasimomentum}

Introducing the total quasimomentum
\begin{displaymath}
K := p+q \in \mathbb{T}^3,
\end{displaymath}
we obtain the von Neumann decomposition
\begin{displaymath}
\mathbb{H}_{\mu_1\mu_2}
= \int_{\mathbb{T}^3}^{\oplus} H_{\mu_1\mu_2}(K)\,\mathrm{d}K.
\end{displaymath}

Define the unitary map
\begin{displaymath}
\Upsilon: L^{2,\mathrm{s}}(\mathbb{T}^3\times\mathbb{T}^3)
\to \int_{\mathbb{T}^3}^{\oplus} L^{2,\mathrm{e}}(\mathbb{T}^3)\,\mathrm{d}K,
\qquad
(\Upsilon f)(K,k)
:= f\!\left(\frac{K}{2}+k,\frac{K}{2}-k\right),
\end{displaymath}
where $k\in\mathbb{T}^3$ is the relative momentum.  
The lattice Schr\"{o}dinger (fiber) operators are given by
\begin{displaymath}
H_{\mu_1\mu_2}(K)=H_0(K)+V_{\mu_1\mu_2}(K),
\end{displaymath}
where
\begin{displaymath}
(H_0(K)f)(p)=\mathcal{E}_K(p)f(p), 
\qquad 
\mathcal{E}_K(p)
= 4\sum_{i=1}^3 \bigl(1 - \cos(K_i/2)\cos p_i\bigr),
\end{displaymath}
and
\begin{equation*}
(V_{\mu_1\mu_2}f)(p)
= \frac{\mu_1}{(2\pi)^3}\int_{\mathbb{T}^3} f(q)\,\mathrm{d} q
+ \frac{2\mu_2}{(2\pi)^3}\sum_{i=1}^3 \cos p_i
    \int_{\mathbb{T}^3} \cos q_i\,f(q)\,\mathrm{d} q.
\end{equation*}
Thus $V_{\mu_1\mu_2}$ is a rank-four self-adjoint operator.

\subsection{Essential spectrum} \label{subsec:ess_spec}

We now determine the essential spectrum of the fiber operators.

\begin{proposition}\label{prop:ess-spectrum}
For every $K\in\mathbb{T}^3$ and $\mu_1,\mu_2\in\mathbb{R}$, the operator $H_{\mu_1\mu_2}(K)$ satisfies
\begin{displaymath}
\sigma_{\mathrm{ess}}\!\left(H_{\mu_1\mu_2}(K)\right)
= \sigma\!\left(H_0(K)\right)
= \bigl[\mathcal{E}_{\min}(K),\,\mathcal{E}_{\max}(K)\bigr],
\end{displaymath}
where
\begin{displaymath}
\mathcal{E}_{\min}(K)
:= \min_{p\in\mathbb{T}^3}\mathcal{E}_K(p), 
\qquad
\mathcal{E}_{\max}(K)
:= \max_{p\in\mathbb{T}^3}\mathcal{E}_K(p).
\end{displaymath}
\end{proposition}

\begin{proof}
Since $V_{\mu_1\mu_2}$ is a rank-four operator, it is compact.  
Hence, by Weyl's theorem,
\begin{displaymath}
\sigma_{\mathrm{ess}}\!\left(H_{\mu_1\mu_2}(K)\right)
= \sigma_{\mathrm{ess}}\!\left(H_0(K)\right).
\end{displaymath}
The operator $H_0(K)$ is multiplication by the continuous function
$\mathcal{E}_K$ on the compact set $\mathbb{T}^3$, so its spectrum is exactly the range of
$\mathcal{E}_K$, i.e.,
\begin{displaymath}
\sigma(H_0(K))
=\sigma_{\mathrm{ess}}(H_0(K))
= [\mathcal{E}_{\min}(K),\mathcal{E}_{\max}(K)].
\end{displaymath}
\end{proof}

\begin{corollary}\label{cor:location-discrete}
For every $K\in\mathbb{T}^3$, all eigenvalues of $H_{\mu_1\mu_2}(K)$ lie outside the spectral
band of $H_0(K)$:
\begin{displaymath}
\sigma_{\mathrm{p}}(H_{\mu_1\mu_2}(K))
\subset (-\infty,\mathcal{E}_{\min}(K))
\cup
(\mathcal{E}_{\max}(K),\infty).
\end{displaymath}
All such eigenvalues are isolated, have finite multiplicity, and may accumulate only at
$\mathcal{E}_{\min}(K)$ or $\mathcal{E}_{\max}(K)$.
\end{corollary}

\begin{proof}
The statement follows from Proposition \ref{prop:ess-spectrum} and the fact that
 $V_{\mu_1\mu_2}$ is the compact self-adjoint operator.
\end{proof}

\section{Invariant subspaces of the fiber Hamiltonian at $K=0$}
\label{sec:InvariantSubspaces}

In this section we describe the natural orthogonal decomposition of 
$L^{2, \mathrm{e}}(\mathbb{T}^3)$ induced by lattice symmetries at zero total
quasimomentum and prove that the fiber Hamiltonian
$H_{\mu_{1}\mu_{2}}(0)$ reduces with respect to this decomposition.

\begin{theorem}\label{Theor:InvariantSub}   
The Hilbert space $ L^{2,\mathrm{e}}(\mathbb{T}^3) $ decomposes orthogonally as
\begin{equation}\label{direct_sum}
L^{2,\mathrm{e}}(\mathbb{T}^3) 
= L^{2,\mathrm{e,s}}(\mathbb{T}^3) \oplus L^{2,\mathrm{e,a_{12}}}(\mathbb{T}^3) \oplus L^{2,\mathrm{e,mix}}(\mathbb{T}^3),
\end{equation}
where
\begin{align*}
L^{2,\mathrm{e,s}}(\mathbb{T}^3)
  &= \{ f\in L^{2,\mathrm{e}}(\mathbb{T}^3): f \ \text{is symmetric under all permutations of } (p_1,p_2,p_3) \},\\
L^{2,\mathrm{e,a_{12}}}(\mathbb{T}^3)
  &= \{ f\in L^{2,\mathrm{e}}(\mathbb{T}^3): f(p_1,p_2,p_3) = -f(p_2,p_1,p_3) \},\\
L^{2,\mathrm{e,mix}}(\mathbb{T}^3)
  &= \Bigl\{ f\in L^{2,\mathrm{e}}(\mathbb{T}^3): f(p_1,p_2,p_3)=f(p_2,p_1,p_3),\ f=f_1+f_2,\\
  &\qquad f_1(p_1,p_2,p_3)=-f_1(p_3,p_2,p_1),\quad 
         f_2(p_1,p_2,p_3)=-f_2(p_1,p_3,p_2) \Bigr\}.
\end{align*}
\end{theorem}

\begin{proof}
For each permutation $\sigma\in S_3$, define the unitary operator
\begin{displaymath}
(U_\sigma f)(p_1,p_2,p_3)=f(p_{\sigma(1)},p_{\sigma(2)},p_{\sigma(3)}),
\end{displaymath}
which preserves $L^{2,\mathrm e}(\mathbb{T}^3)$. Let $U_{12},U_{13},U_{23}$ be the transpositions.
 We introduce
\begin{displaymath}
P_{\mathrm{s}}=\frac{1}{6}\sum_{\sigma\in S_3} U_{\sigma},\quad
P_{\mathrm{a_{12}}}=\frac{1}{2}(I-U_{12}),\quad
P_{\mathrm{mix}}=I-P_{\mathrm{s}}-P_{\mathrm{a_{12}}}.
\end{displaymath}
The operators $P_{\mathrm{s}}$ and $P_{\mathrm{a_{12}}}$ are orthogonal projections and satisfy
$P_{\mathrm{s}}P_{\mathrm{a_{12}}}=0$. Hence $P_{\mathrm{mix}}$ is also an orthogonal projection and
\begin{displaymath}
P_{\mathrm{s}}+P_{\mathrm{a_{12}}}+P_{\mathrm{mix}}=I.
\end{displaymath}

Now, we show that the range of the projection $P_{\theta}$ coincides with the subspace $L^{2,\mathrm{e, \theta}}(\mathbb{T}^3),\,\theta\in\{\mathrm{s, a_{12}, mix}\}.$
By construction,
\begin{displaymath}
\mathrm{Ran} P_{\mathrm{s}} = L^{2,\mathrm{e,s}}(\mathbb{T}^3),\quad
\mathrm{Ran}P_{\mathrm{a_{12}}} = L^{2,\mathrm{e,a_{12}}}(\mathbb{T}^3).
\end{displaymath}
Let $g\in\mathrm{Ran}\,P_{\mathrm{mix}}$. Then $P_{\mathrm{s}}g=0$ and $U_{12}g=g$ which implies $g+U_{13}g+U_{23}g=0$.
Define
\begin{displaymath}
g^{(13)}=\frac{1}{3}(I-U_{13})g,\qquad
g^{(23)}=\frac{1}{3}(I-U_{23})g.
\end{displaymath}
Then $g^{(13)}$ is antisymmetric in $(p_1,p_3)$, $g^{(23)}$ is antisymmetric in $(p_2,p_3)$, and
$
g=g^{(13)}+g^{(23)}.
$
Thus $\mathrm{Ran}P_{\mathrm{mix}} = L^{2,\mathrm{e,mix}}(\mathbb{T}^3)$.

Since the three projections are mutually orthogonal,
their ranges are mutually orthogonal.  
From 
$
P_{\mathrm{s}}+P_{\mathrm{a_{12}}}+P_{\mathrm{mix}}=I
$
we obtain for every $f\in L^{2,\mathrm e}(\mathbb{T}^3)$:
\begin{displaymath}
f=P_{\mathrm{s}}f+P_{\mathrm{a_{12}}}f+P_{\mathrm{mix}}f,
\end{displaymath}
which gives the decomposition \eqref{direct_sum}, and uniqueness follows from orthogonality.
\end{proof}

\begin{lemma}\label{invariance} 
Each subspace in \eqref{direct_sum} is invariant under the operator 
$H_{\mu_1\mu_2}(0)$, and hence reduces it.
\end{lemma}

\begin{proof}
Since $H_{\mu_1\mu_2}(0)=H_0(0)+V_{\mu_1\mu_2}$, it suffices to show that the
three subspaces are invariant under both terms.
First, since the operator $H_0(0)$ is the multiplication operator by the fully symmetric function
$\mathcal{E}_0(p_1,p_2,p_3)$, each subspace is invariant under $H_0(0)$.
Next,
for $f\in L^{2,\mathrm{e}}(\mathbb{T}^3)$,
\begin{displaymath}
(V_{\mu_1\mu_2}f)(p)
= \frac{\mu_1}{8\pi^{3}}\!\int_{\mathbb{T}^3} f(q)\,\mathrm{d} q
 +\frac{\mu_2}{4\pi^{3}}\sum_{i=1}^{3} \cos p_i
   \!\int_{\mathbb{T}^3}\! \cos q_i\, f(q)\,\mathrm{d} q.
\end{displaymath}
\textit{(i) Symmetric subspace.}
If $f\in L^{2,\mathrm{e,s}}(\mathbb{T}^3)$, then the integrals
$\int_{\mathbb{T}^3} \cos q_i f(q)\,\mathrm{d} q$ are equal for $i=1,2,3$.
Hence
\begin{displaymath}
(V_{\mu_1\mu_2}f)(p)
= c_0 + c_1(\cos p_1+\cos p_2+\cos p_3),
\end{displaymath}
which is fully symmetric; thus the subspace is invariant.

\textit{(ii) Antisymmetric subspace.}
If $f\in L^{2,\mathrm{e,a_{12}}}(\mathbb{T}^3)$, then
\begin{displaymath}
\int_{\mathbb{T}^3} g(q)f(q)\,\mathrm{d} q = 0,\quad g\in\{1, \cos q_3, \cos q_1+\cos q_2\}.
\end{displaymath}
Hence
\begin{displaymath}
(V_{\mu_1\mu_2}f)(p)
   = c(\cos p_1-\cos p_2),
\end{displaymath}
which is antisymmetric in $(p_1,p_2)$.  
Thus $L^{2,\mathrm{e,a_{12}}}(\mathbb{T}^3)$ is invariant.

\textit{(iii) Mixed subspace.}
Let $f\in L^{2,\mathrm{e,mix}}(\mathbb{T}^3)$.  
Then $f=f_1+f_2$, where $f$ is symmetric in $(p_1,p_2)$ and
\begin{displaymath}
f_1(p_1,p_2,p_3)=-f_1(p_3,p_2,p_1),\qquad
f_2(p_1,p_2,p_3)=-f_2(p_1,p_3,p_2).
\end{displaymath}
A computation analogous to part (ii) yields
\begin{displaymath}
(V_{\mu_1\mu_2}f)(p)
   = c\bigl(\cos p_1+\cos p_2 - 2\cos p_3\bigr),
\end{displaymath}
where the coefficient $c$ is the common value of the integrals
$\int\cos q_1 f_1(q) \,\mathrm{d} q \\
= \int\cos q_2 f_2(q) \,\mathrm{d} q$.
The resulting function is symmetric in $(p_1,p_2)$ and decomposes into
terms antisymmetric in $(p_1,p_3)$ and $(p_2,p_3)$, so it lies in
$L^{2,\mathrm{e,mix}}(\mathbb{T}^3)$.

Thus all three subspaces are invariant under $H_0(0)$ and $V_{\mu_1\mu_2}$,
and because they form an orthogonal direct sum of $L^{2,\mathrm{e}}(\mathbb{T}^3)$,
they reduce $H_{\mu_1\mu_2}(0)$.
\end{proof}

\begin{lemma}\label{lem:orth-basis}
The functions
\begin{displaymath}
\alpha^{\mathrm{s}}_{1}(p)=\frac{1}{\sqrt{8\pi^3}},\qquad
\alpha^{\mathrm{s}}_{2}(p)=\frac{\cos p_{1}+\cos p_{2}+\cos p_{3}}{\sqrt{12\pi^3}},
\end{displaymath}
\begin{displaymath}
\alpha^{\mathrm{a_{12}}}_{3}(p)=\frac{\cos p_1-\cos p_2}{\sqrt{8\pi^3}},\qquad
\alpha^{\mathrm{mix}}_{3}(p)=\frac{\cos p_1+\cos p_2-2\cos p_3}{\sqrt{24\pi^3}},
\end{displaymath}
form an orthonormal set in $L^{2,\mathrm{e}}(\mathbb{T}^3)$, with
\begin{displaymath}
\{\alpha^{\mathrm{s}}_1,\alpha^{\mathrm{s}}_2\}\subset L^{2,\mathrm{e,s}}(\mathbb{T}^3),\quad
\alpha^{\mathrm{a_{12}}}_3\in L^{2,\mathrm{e,a_{12}}}(\mathbb{T}^3),\quad
\alpha^{\mathrm{mix}}_3\in L^{2,\mathrm{e,mix}}(\mathbb{T}^3).
\end{displaymath}
\end{lemma}

\begin{proof}
Using
\begin{displaymath}
\int_{\mathbb{T}^3}1\,\mathrm{d}p=8\pi^3,\qquad 
\int_{\mathbb{T}^3}\cos p_i\,\mathrm{d}p=0,\qquad
\int_{\mathbb{T}^3}\cos p_i\cos p_j\,\mathrm{d}p=
\begin{cases}
4\pi^3,&i=j,\\
0,&i\neq j,
\end{cases}
\end{displaymath}
one checks directly that all four functions have norm $1$ and that their pairwise inner products vanish.

The membership in the symmetry subspaces follows from the obvious parity properties:  
$\alpha^{\mathrm{s}}_1,\alpha^{\mathrm{s}}_2$ are symmetric in all coordinates;  
$\alpha^{\mathrm{a_{12}}}_3$ changes sign under $p_1\leftrightarrow p_2$;  
$\alpha^{\mathrm{mix}}_3$ is symmetric in $(p_1,p_2)$ and orthogonal to the previous vectors.
\end{proof}

\begin{lemma}\label{lem:V-decompos-invariance}
Let $\alpha^{\mathrm{s}}_{1},\alpha^{\mathrm{s}}_{2},
\alpha^{\mathrm{a_{12}}}_{3},\alpha^{\mathrm{mix}}_{3}$ be the vectors from 
Lemma \ref{lem:orth-basis}.  
Then:

(i) The operator $V_{\mu_1\mu_2}$ is a rank-four self-adjoint operator and admits the representation
\begin{equation}\label{eq:V-rank-one}
V_{\mu_1\mu_2}
=\mu_1(\cdot,\alpha^{\mathrm{s}}_{1})\alpha^{\mathrm{s}}_{1}
+\mu_2(\cdot,\alpha^{\mathrm{s}}_{2})\alpha^{\mathrm{s}}_{2}
+\mu_2(\cdot,\alpha^{\mathrm{a_{12}}}_{3})\alpha^{\mathrm{a_{12}}}_{3}
+\mu_2(\cdot,\alpha^{\mathrm{mix}}_{3})\alpha^{\mathrm{mix}}_{3}.
\end{equation}

(ii)  With the orthogonal decomposition
\begin{displaymath}
L^{2,\mathrm{e}}(\mathbb{T}^3)
=L^{2,\mathrm{e,s}}(\mathbb{T}^3)\oplus L^{2,\mathrm{e,a_{12}}}(\mathbb{T}^3)\oplus L^{2,\mathrm{e,mix}}(\mathbb{T}^3),
\end{displaymath}
 the restrictions of $V_{\mu_1\mu_2}$ are
\begin{align}\label{eq:V-restrictions_2}
V^{\mathrm{s}}_{\mu_1\mu_2}
&=\mu_1(\cdot,\alpha^{\mathrm{s}}_{1})\alpha^{\mathrm{s}}_{1}
+\mu_2(\cdot,\alpha^{\mathrm{s}}_{2})\alpha^{\mathrm{s}}_{2},\\
V^{\mathrm{a_{12}}}_{\mu_2}
&=\mu_2(\cdot,\alpha^{\mathrm{a_{12}}}_{3})\alpha^{\mathrm{a_{12}}}_{3},\nonumber\\
V^{\mathrm{mix}}_{\mu_2}
&=\mu_2(\cdot,\alpha^{\mathrm{mix}}_{3})\alpha^{\mathrm{mix}}_{3}.\nonumber
\end{align}

Consequently,
\begin{align*}
H^{\mathrm{s}}_{\mu_1\mu_2}(0)&=H_0(0)+V^{\mathrm{s}}_{\mu_1\mu_2},\nonumber\\
H^{\mathrm{a_{12}}}_{\mu_2}(0)&=H_0(0)+V^{\mathrm{a_{12}}}_{\mu_2},\\
H^{\mathrm{mix}}_{\mu_2}(0)&=H_0(0)+V^{\mathrm{mix}}_{\mu_2}.\nonumber
\end{align*}
\end{lemma}

\begin{proof}
The identity
\begin{align*}
6\sum\limits_{j=1}^{3}\cos p_j\cos q_j
&=2\Big(\sum_{j=1}^{3}\cos p_j\Big)\Big(\sum_{j=1}^{3}\cos q_j\Big)
+3(\cos p_1-\cos p_2)(\cos q_1-\cos q_2)\\
&+(\cos p_1+\cos p_2-2\cos p_3)(\cos q_1+\cos q_2-2\cos q_3)
\end{align*}
shows that $V_{\mu_1\mu_2}$ is a linear combination of the four rank-one operators
$(f,\alpha)\alpha$ for the listed basis vectors, giving \eqref{eq:V-rank-one}.
Since the vectors are orthonormal, the operator is self-adjoint with rank of four.

The formulas in \eqref{eq:V-restrictions_2} follow by keeping only the terms belonging to each symmetry subspace.  
As $H_0(0)$ preserves the decomposition, the same holds for $H_{\mu_1\mu_2}(0)=H_0(0)+V_{\mu_1\mu_2}$.
\end{proof}
 
\begin{corollary}\label{cor:operator_decomposition} 
With respect to the orthogonal decomposition \eqref{direct_sum},  
the operator $H_{\mu_1\mu_2}(0)$ admits the following direct sum decomposition:
\begin{equation*}
H_{\mu_1\mu_2}(0)
= H^{\mathrm{s}}_{\mu_1\mu_2}(0) 
\oplus H^{\mathrm{a_{12}}}_{\mu_2}(0) 
\oplus H^{\mathrm{mix}}_{\mu_2}(0).
\end{equation*}
\end{corollary}

\begin{corollary}\label{corollary_full_spectr}
The spectrum of the operator $H_{\mu_1\mu_2}(0)$ satisfies the equality
\begin{equation*}
\sigma\bigl(H_{\mu_1\mu_2}(0)\bigr)
=
\sigma\left( H^{\mathrm{s}}_{\mu_1\mu_2}(0) \right)
\cup 
\sigma\left( H^{\mathrm{a_{12}}}_{\mu_2}(0) \right)
\cup 
\sigma\left( H^{\mathrm{mix}}_{\mu_2}(0) \right).
\end{equation*}
Hence, the spectral analysis of $H_{\mu_1\mu_2}(0)$ reduces to the independent study of its restrictions to the subspaces 
$ L^{2,\mathrm{e,s}}(\mathbb{T}^3) $, $ L^{2,\mathrm{e,a_{12}}}(\mathbb{T}^3) $, and $ L^{2,\mathrm{e,mix}}(\mathbb{T}^3) $.
\end{corollary}

The statements of Corollaries \ref{cor:operator_decomposition} and \ref{corollary_full_spectr} follow directly from standard results on the decomposition of bounded linear operators with respect to invariant orthogonal subspaces.

Now we will study the spectrum of the operators $H^{\mathrm{s}}_{\mu_1\mu_2}(0)$ and 
$H^{\theta}_{\mu_2}(0),\,\theta\in\{\mathrm{a_{12}}, \mathrm{mix}\}$.

\subsection{Lippmann--Schwinger operator and eigenvalue determinants}

\begin{definition}[Transpose Lippmann--Schwinger operator]\label{def:LS_operator}
For $z \in \mathbb{C} \setminus [0,24]$, the \emph{transpose} Lippmann--Schwinger operators associated with the symmetric and 
one-parameter interaction potentials are defined by
\begin{displaymath}
B^{\mathrm{s}}_{\mu_1\mu_2}(0,z)
   := - V^{\mathrm{s}}_{\mu_1\mu_2}\, R_0(0,z), 
   \qquad
B^{\theta}_{\mu_2}(0,z)
   := - V^{\theta}_{\mu_2}\, R_0(0,z),
   \quad \theta\in\{\mathrm{a_{12}},\mathrm{mix}\},
\end{displaymath}
where 
\begin{displaymath}
R_0(0,z) := (H_0(0)-zI)^{-1}
\end{displaymath}
is the free resolvent of the lattice Schr\"{o}dinger operator (cf.\ the classical Lippmann--Schwinger construction \cite{LSchwinger:1950}).  
The terminology “transpose” reflects that the operators above act on the left of the free resolvent, in contrast to the conventional Lippmann--Schwinger form $R_0 V$.
\end{definition}

\begin{lemma}[Integral representation and determinant criterion]\label{lem:LS_and_det}
Let $f\in L^{2,\mathrm{e}}\\ (\mathbb{T}^3)$ and $z\in\mathbb{C}\setminus[0,24]$.  
Then the operators $B^{\mathrm{s}}_{\mu_1\mu_2}$ and $B^{\theta}_{\mu_2}$ admit the integral representations
\begin{align}
(B^{\mathrm{s}}_{\mu_1\mu_2}f)(p)
&= -\mu_1 \int_{\mathbb{T}^3}\frac{\alpha^{\mathrm{s}}_1(p)
\alpha^{\mathrm{s}}_1(q)f(q)}{\mathcal{E}_0(q)-z} \mathrm{d}q-\mu_2 \int_{\mathbb{T}^3}\frac{\alpha^{\mathrm{s}}_2(p)
\alpha^{\mathrm{s}}_2(q)f(q)}{\mathcal{E}_0(q)-z}\mathrm{d}q, 
\label{B_s_integral2}
\\
(B^{\theta}_{\mu_2}f)(p)
&= -\mu_2 \int_{\mathbb{T}^3}
\frac{\alpha^{\theta}_3(p)\alpha^{\theta}_3(q)\, f(q)}{\mathcal{E}_0(q)-z}\, \mathrm{d}q,
\qquad \theta\in\{\mathrm{a_{12}},\mathrm{mix}\},
\label{B_theta_integral2}
\end{align}
where the functions $\alpha^{\mathrm{s}}_i$, $i=1,2$, and $\alpha^{\theta}_3$, 
$\theta\in\{\mathrm{a_{12}},\mathrm{mix}\}$, are given in Lemma \ref{lem:orth-basis}.

Moreover, $z$ is an eigenvalue of $H^{\mathrm{s}}_{\mu_1\mu_2}(0)$ if and only if
\begin{displaymath}
\Delta^{\mathrm{s}}_{\mu_1\mu_2}(z)
:=\det\bigl[I - B^{\mathrm{s}}_{\mu_1\mu_2}(0,z)\bigr]
= 
\begin{vmatrix}
1+\mu_1 a^{\mathrm{s}}_{11}(z) & \mu_2 a^{\mathrm{s}}_{12}(z) \\
\mu_1 a^{\mathrm{s}}_{12}(z) & 1+\mu_2 a^{\mathrm{s}}_{22}(z)
\end{vmatrix}
= 0.
\end{displaymath}
Similarly, $z$ is an eigenvalue of $H^{\theta}_{\mu_2}(0)$, $\theta\in\{\mathrm{a_{12}},\mathrm{mix}\}$, if and only if
\begin{displaymath}
\Delta^{\theta}_{\mu_2}(z)
:= \det\bigl[I - B^{\theta}_{\mu_2}(0,z)\bigr]
= 1+\mu_2 a^{\theta}(z) = 0.
\end{displaymath}

Here
\begin{equation}\label{a_{ij}}
a^{\mathrm{s}}_{ij}(z)
= \int_{\mathbb{T}^3}\frac{\alpha^{\mathrm{s}}_i(p)
\alpha^{\mathrm{s}}_j(p)}{\mathcal{E}_0(p)-z}\mathrm{d}p,
\qquad
a^{\theta}(z)
= \int_{\mathbb{T}^3}\frac{\bigl(\alpha^{\theta}_3(p)\bigr)^2}{\mathcal{E}_0(p)-z}\mathrm{d}p,
\end{equation}
and the function $a^{\theta}(\cdot)$ is independent of $\theta$.
\end{lemma}

\begin{proof}The integral representations \eqref{B_s_integral2}--\eqref{B_theta_integral2} follow directly from Lemmas \ref{lem:orth-basis} and \ref{lem:V-decompos-invariance}. 

For eigenvalues, the Lippmann--Schwinger equation
\begin{displaymath}
[I-B]\varphi=0
\end{displaymath}
reduces to a $2\times 2$ system in the symmetric subspace and a scalar equation in the rank-one subspaces, giving the determinant formulas.  

Equality $ a^{\mathrm{a_{12}}}(z) = a^{\mathrm{mix}}(z) =: a(z) $ follows by expanding
\begin{displaymath}
(\cos p_1 - \cos p_2)^2 \quad\text{and}\quad (\cos p_1 + \cos p_2 - 2\cos p_3)^2
\end{displaymath}
and using fully symmetric property of $\mathcal{E}_0(p)$.
\end{proof}

\begin{proposition}\label{lem:prop_function}
Let $a^{\mathrm{s}}_{ij}(z)$, $i,j=1,2$, and 
$a(z):=a^{\theta}(z)$, $\theta\in\{\mathrm{a}_{12},\mathrm{mix}\}$, be defined by
\eqref{a_{ij}}. Then the following statements hold:

\begin{enumerate}
\item[(a)] 
All functions $a^{\mathrm{s}}_{ij}(z)=a^{\mathrm{s}}_{ji}(z)$ and 
$a(z)$ are real-valued and analytic on 
$\mathbb{R}\setminus[0,24]$.

\item[(b)]  
They are strictly increasing on each of the intervals 
$(-\infty,0]$ and $[24,+\infty)$, satisfy
\begin{displaymath}
a^{\mathrm{s}}_{ij}(z)>0,\qquad a(z)>0,\qquad z\leq 0,
\end{displaymath}
and
\begin{displaymath}
a^{\mathrm{s}}_{ij}(z)<0,\qquad a(z)<0,\qquad z\geq 24.
\end{displaymath}

\item[(c)]  
The following asymptotic expansions hold:
\begin{displaymath}
a^{\mathrm{s}}_{ij}(z)=
\begin{cases}
a^{\mathrm{s}}_{ij}(0)+O(\sqrt{-z}), & z\nearrow 0,\\
(-1)^{i+j+1}\!\bigl(a^{\mathrm{s}}_{ij}(0)+O(\sqrt{z-24})\bigr),
& z\searrow 24,
\end{cases}
\end{displaymath}
and
\begin{displaymath}
a(z)=
\begin{cases}
a(0)+O(-z), & z\nearrow 0,\\
-a(0)+O(z-24), & z\searrow 24.
\end{cases}
\end{displaymath}

\item[(d)]  
The functions $a^{\mathrm{s}}_{ij}(z)$ satisfy, for all $z\in\mathbb{R}\setminus(0,24)$,
\begin{displaymath}
a^{\mathrm{s}}_{12}(z)=\frac{12-z}{2\sqrt{6}}\,a^{\mathrm{s}}_{11}(z)
-\frac{1}{2\sqrt{6}},
\qquad
a^{\mathrm{s}}_{22}(z)=\frac{12-z}{2\sqrt{6}}\,a^{\mathrm{s}}_{12}(z).
\end{displaymath}

\item[(e)]  
The inequality
\begin{displaymath}
a^{\mathrm{s}}_{11}(0)>\frac{11}{102}
\end{displaymath}
holds.
\end{enumerate}
\end{proposition}

\begin{proof}
\textbf{(a).}
All functions under consideration are of the form
\begin{displaymath}
\int_{\mathbb{T}^{3}}
\frac{g(p)}{\mathcal{E}_{0}(p)-z}\,\mathrm{d}p,
\end{displaymath}
where $g$ is a trigonometric polynomial.  
Since $\mathcal{E}_{0}(p)-z\neq 0$ for $z\notin[0,24]$, the integrand is analytic in $z$, uniformly in $p$.  
Analyticity follows from dominated convergence.

\medskip
\textbf{(b).}
We prove the statements for $a^{\mathrm{s}}_{12}(z)$, $z\leq 0$; the other cases are identical.

Let $z_{1}<z_{2}\leq 0$. Then
\begin{displaymath}
a^{\mathrm{s}}_{12}(z_2)-a^{\mathrm{s}}_{12}(z_1)
=(z_2-z_1)\int_{\mathbb{T}^{3}}
\frac{3\cos p_{1}}
{4\sqrt{6}\pi^{3}
(\mathcal{E}_{0}(p)-z_1)(\mathcal{E}_{0}(p)-z_2)}\mathrm{d}p.
\end{displaymath}
Write $u=12-4\cos p_{2}-4\cos p_{3}$ and integrate first in $p_{1}$.  
A direct computation shows that
\begin{align*}
&\int_{-\pi}^{\pi}
\frac{\cos p_{1}}
{(u-z_1-4\cos p_1)(u-z_2-4\cos p_1)}\mathrm{d}p_1\\
&= \int_{0}^{\pi}\frac{4\cos^{2}p_1(2u-z_1-z_2)}{((u-z_1)^{2}-16\cos^{2}p_1)((u-z_2)^{2}-16\cos^{2}p_1)}\mathrm{d}p_1.
\end{align*}
Since $u-z_{i}\geq 4$ for $z_i\leq 0$, the integrand above is positive, implying that
$a^{\mathrm{s}}_{12}(z)$ is strictly increasing on $(-\infty,0]$.  
The sign of $a^{\mathrm{s}}_{12}(z)$ for $z\leq 0$ follows in the same way.

\textbf{(c).}
We present the proof for $a^{\mathrm{s}}_{11}(z)$ as $z\nearrow 0$.  
The other cases are similar.

Split the integral into a neighborhood of the unique minimum of $\mathcal{E}_{0}$,
\begin{displaymath}
a^{(1)}_{11}(z)=\frac{1}{8\pi^{3}}\int_{U_{\delta}(0)}
\frac{\mathrm{d}p}{\mathcal{E}_{0}(p)-z},
\quad
a^{(2)}_{11}(z)=\frac{1}{8\pi^{3}}\int_{\mathbb{T}^{3}\setminus U_{\delta}(0)}
\frac{\mathrm{d}p}{\mathcal{E}_{0}(p)-z}.
\end{displaymath}
Since $\mathcal{E}_{0}$ has a non-degenerate minimum at $0$,  
$a^{(1)}_{11}(z)$ gives the non-analytic term $O(\sqrt{-z})$,  
while $a^{(2)}_{11}(z)$ is analytic near $z=0$.  
This is a standard argument (see, e.g., Lemma 3 in \cite{Lakaev:1992}).  
The expansions at $z=24$ follow similarly from a neighborhood of the maximum.

\textbf{(d).}
From the definitions of $\alpha^{\mathrm{s}}_{1}$ and $\alpha^{\mathrm{s}}_{2}$,
\begin{displaymath}
\alpha^{\mathrm{s}}_{2}(p)=\frac{12-\mathcal{E}_{0}(p)}{2\sqrt{6}}\,
\alpha^{\mathrm{s}}_{1}(p).
\end{displaymath}
Multiplying, integrating, and using the orthogonality yields the stated formulas.

\textbf{(e).}
The inequality
$a^{\mathrm{s}}_{11}(0)>\frac{11}{102}$
is established in Lemma 6 of \cite{ALMM:2006}.
\end{proof}
\begin{lemma}\label{lem:polynomial_A_pm}
Let $\mu_1, \mu_2 \in \mathbb{R}$.  
The functions $\Delta^{\mathrm{s}}_{\mu_1\mu_2}(z)$ and 
$\Delta_{\mu_2}(z):=\Delta^{\theta}_{\mu_2}(z)$ are real-valued and analytic on 
$z \in \mathbb{R}\setminus[0,24]$. Moreover, the following asymptotic properties hold:

\begin{enumerate}
\item[(i)] 
\begin{displaymath}
\lim_{|z|\to \infty} \Delta^{\mathrm{s}}_{\mu_1\mu_2}(z) = 1, 
\qquad 
\lim_{|z|\to \infty} \Delta_{\mu_2}(z) = 1.
\end{displaymath}

\item[(ii)] 
As $z$ approaches the edges of the essential spectrum, we have
\begin{displaymath}
\Delta^{\mathrm{s}}_{\mu_1 \mu_2}(z) =
\begin{cases}
A^{-}(\mu_1,\mu_2) + O(\sqrt{-z}), & z \nearrow 0,\\
A^{+}(\mu_1,\mu_2) + O(\sqrt{z-24}), & z \searrow 24,
\end{cases}
\end{displaymath}
where
\begin{equation}\label{def:polynomial_A_pm}
A^{\mp}(\mu_1,\mu_2) 
:= 1 \pm \bigl(a^{\mathrm{s}}_{11}(0)\mu_1 + a^{\mathrm{s}}_{22}(0)\mu_2\bigr) 
+ \bigl(a^{\mathrm{s}}_{11}(0)a^{\mathrm{s}}_{22}(0) - (a^{\mathrm{s}}_{12}(0))^2\bigr)\mu_1\mu_2.
\end{equation}

\item[(iii)] 
Similarly, for the single-parameter determinant,
\begin{displaymath}
\Delta_{\mu_2}(z) =
\begin{cases}
1 + \mu_2\, a(0) + O(-z), & z \nearrow 0,\\
1 - \mu_2\, a(0) + O(z-24), & z \searrow 24,
\end{cases}
\end{displaymath}
where $a(0)=a^{\mathrm{a}_{12}}(0)=a^{\mathrm{mix}}(0)$.
\end{enumerate}
\end{lemma}

\begin{proof}
Proposition \ref{lem:prop_function},(a) implies the functions 
$\Delta^{\mathrm{s}}_{\mu_1\mu_2}(z)$ and $\Delta_{\mu_2}(z)$ are real-valued and analytic on 
$\mathbb{R}\setminus[0,24]$.
Moreover, since $a^{\mathrm{s}}_{ij}(z),\, a(z) = O(1/|z|)$ as $|z|\to\infty$, we obtain
\begin{displaymath}
\Delta^{\mathrm{s}}_{\mu_1\mu_2}(z)\to 1, 
\qquad 
\Delta_{\mu_2}(z)\to 1.
\end{displaymath}

\medskip
The proof of items (ii) and (iii) follows directly from Proposition \ref{lem:prop_function},(c) and the definitions of $\Delta^{\mathrm{s}}_{\mu_1\mu_2}$ and $\Delta^{\theta}_{\mu_2}$.
\end{proof}

\section{The summary of the key findings and supporting lemmas}\label{sec:MainResults}

\begin{lemma}\label{lem:regions_A_pm} 
Define
\begin{displaymath}
\mu_2^{-}(\mu_1)=\frac{24}{\mu_1+12}-\mu_0,\qquad \mu_1\neq -12,
\end{displaymath}
\begin{displaymath}
\mu_2^{+}(\mu_1)=\frac{24}{\mu_1-12}+\mu_0,\qquad \mu_1\neq 12,
\end{displaymath}
where
\begin{displaymath}
\mu_0=\frac{24\,a^{\mathrm s}_{11}(0)}{12\,a^{\mathrm s}_{11}(0)-1},
\qquad 
12 a^{\mathrm s}_{11}(0)-1>0.
\end{displaymath}
Then the zero sets of the polynomials $A^{-}(\mu_1,\mu_2)$ and
$A^{+}(\mu_1,\mu_2)$ coincide with the graphs of the functions
$\mu_2^{-}(\mu_1)$ and $\mu_2^{+}(\mu_1)$, respectively.

Each graph is a smooth hyperbola with two unbounded, connected
components:
\begin{displaymath}
\tau^{-}_0=\{(\mu_1,\mu_2):\mu_1>-12,\ \mu_2=\mu_2^{-}(\mu_1)\},\,
\tau^{-}_1=\{(\mu_1,\mu_2):\mu_1<-12,\ \mu_2=\mu_2^{-}(\mu_1)\},
\end{displaymath}
\begin{displaymath}
\tau^{+}_0=\{(\mu_1,\mu_2):\mu_1<12,\ \mu_2=\mu_2^{+}(\mu_1)\},\quad
\tau^{+}_1=\{(\mu_1,\mu_2):\mu_1>12,\ \mu_2=\mu_2^{+}(\mu_1)\}.
\end{displaymath}
Moreover, each pair $(\tau_0^{-},\tau_1^{-})$ divides the
$(\mu_1,\mu_2)$-plane into three open, unbounded, connected regions:
\begin{align*}
\mathcal{A}_0^{-}=&\{(\mu_1,\mu_2):\mu_1>-12,\ \mu_2>\mu_2^{-}(\mu_1)\},
\\\nonumber
\mathcal{A}_1^{-}=&
\{(\mu_1,\mu_2):\mu_1>-12,\ \mu_2<\mu_2^{-}(\mu_1)\}
\cup\{(\mu_1,\mu_2):\mu_1<-12,\ \mu_2>\mu_2^{-}(\mu_1)\}
\cup\\\nonumber
&\{(-12,\mu_2):\mu_2\in\mathbb{R}\},
\nonumber\\
\mathcal{A}_2^{-}=&\{(\mu_1,\mu_2):\mu_1<-12,\ \mu_2<\mu_2^{-}(\mu_1)\}.
\nonumber
\end{align*}
Analogously, each pair $(\tau_0^{+},\tau_1^{+})$ divides the
$(\mu_1,\mu_2)$-plane into three open, unbounded, connected regions:
\begin{align*}
\mathcal{A}_0^{+}=&\{(\mu_1,\mu_2):\mu_1<12,\ \mu_2<\mu_2^{+}(\mu_1)\},
\label{def:regions_A_plus}\\\nonumber
\mathcal{A}_1^{+}=&
\{(\mu_1,\mu_2):\mu_1<12,\ \mu_2>\mu_2^{+}(\mu_1)\}
\cup\{(\mu_1,\mu_2):\mu_1>12,\ \mu_2<\mu_2^{+}(\mu_1)\}
\cup\\\nonumber
&\{(12,\mu_2):\mu_2\in\mathbb{R}\},
\nonumber\\
\mathcal{A}_2^{+}=&\{(\mu_1,\mu_2):\mu_1>12,\ \mu_2>\mu_2^{+}(\mu_1)\}.
\nonumber
\end{align*}
Hence the triples $(\mathcal{A}_0^{\mp},\mathcal{A}_1^{\mp},\mathcal{A}_2^{\mp})$ ar pairwise disjoint and, together with $\tau^{\mp}_0$ and $\tau^{\mp}_1$, cover $\mathbb{R}^2$.
\end{lemma}

\begin{proof}[Proof of Lemma \ref{lem:regions_A_pm}]
We treat the "$-$" case; the "$+$" case is identical.
\textit{Zero set of $A^{-}(\mu_1,\mu_2)$.}
Using Proposition \ref{lem:prop_function},(d), the polynomial $A^{-}(\mu_1,\mu_2)$ given by \eqref{def:polynomial_A_pm} can be rewritten as
\begin{equation}\label{polynomial_A}
A^{-}(\mu_1,\mu_2)
=\frac{a^{\mathrm{s}}_{11}(0)}{\mu_0}\bigl[(\mu_2+\mu_0)(\mu_1+12)-24\bigr],
\qquad 
\mu_0=\frac{24\,a^{\mathrm{s}}_{11}(0)}{12\,a^{\mathrm{s}}_{11}(0)-1}.
\end{equation}
Since $a^{\mathrm{s}}_{11}(0)>\frac{11}{102}$ (Proposition  \ref{lem:prop_function},(e)),
we have $\mu_0>2$, and thus $A^{-}(\mu_1,\mu_2)=0$ is equivalent to
\begin{equation}\label{eq:hyperbola}
(\mu_1+12)(\mu_2+\mu_0)=24.
\end{equation}
Solving \eqref{eq:hyperbola} for $\mu_2$ gives
\begin{displaymath}
\mu_2=\mu_2^{-}(\mu_1)=\frac{24}{\mu_1+12}-\mu_0,
\qquad \mu_1\neq -12.
\end{displaymath}
Hence the zero set of $A^{-}$ is precisely the graph of $\mu_2^{-}(\mu_1)$.
This hyperbola has two unbounded, connected components $\tau_0^{-}$ and $\tau_1^{-}$,
and it divides $\mathbb{R}^2$ into three open, unbounded, connected regions.

Inspecting the sign of the expression in \eqref{polynomial_A} in each
region shows that these three regions are exactly
$\mathcal{A}_0^{-}$, $\mathcal{A}_1^{-}$, and $\mathcal{A}_2^{-}$.
Their union is
\begin{displaymath}
\mathbb{R}^2 \setminus (\tau_0^{-}\cup\tau_1^{-}),
\end{displaymath}
and they are mutually disjoint. 
\end{proof}

\begin{lemma}\label{lem:regions_B}
Let
\begin{displaymath}
\mu_2^{(0)}=\frac{1}{a(0)}.
\end{displaymath}
The point $-\mu_2^{(0)}$ divides $\mu_2$-axis into two disjoint open intervals:
\begin{displaymath}
\mathcal B_0^-=\{\mu_2>-\,\mu_2^{(0)}\},\qquad
\mathcal B_1^-=\{\mu_2<-\,\mu_2^{(0)}\},
\end{displaymath}
and the point $\mu_2^{(0)}$ divides it into
\begin{displaymath}
\mathcal B_0^+=\{\mu_2<\mu_2^{(0)}\},\qquad
\mathcal B_1^+=\{\mu_2>\mu_2^{(0)}\}.
\end{displaymath}
\end{lemma}

\begin{proof}
By Proposition \ref{lem:prop_function},(b),
$\mu_2^{(0)}=1/a(0)$ is finite and nonzero.  
Thus, $-\mu_2^{(0)}$ divides $\mathbb{R}$ into two open intervals,
which correspond to $\mathcal B_0^{-}$ and $\mathcal B_1^{-}$.
Likewise, $\mu_2^{(0)}$ divides $\mathbb{R}$ into the open intervals
defining $\mathcal B_0^{+}$ and $\mathcal B_1^{+}$.
\end{proof}

\begin{theorem}\label{theo:constant_1}
Let $(\mu_1,\mu_2)\in\mathcal{A}$, where $\mathcal{A}=\mathcal{A}_{\zeta}^{-}$
(resp.\ $\mathcal{A}=\mathcal{A}_{\zeta}^{+}$) for some $\zeta\in\{0,1,2\}$.
Then $H^{\mathrm{s}}_{\mu_1\mu_2}(0)$ has exactly $\zeta$ eigenvalues, counted
with multiplicity, lying below (resp.\ above) its essential spectrum
$\sigma_{\mathrm{ess}}(H_{\mu_1\mu_2}(0))$.

For any such eigenvalue $z=z(\mu_1,\mu_2)$, an eigenfunction can be chosen as
\begin{equation}\label{rep_eigenfunction}
f^{\mathrm{s}}_{\mu_1\mu_2}(p)
= C \left(
\mu_1\mu_2\, a^{\mathrm{s}}_{12}(z)
\frac{\alpha^{\mathrm{s}}_1(p)}{\mathcal{E}_0(p)-z}
-
\mu_2\bigl(1+\mu_1 a^{\mathrm{s}}_{11}(z)\bigr)
\frac{\alpha^{\mathrm{s}}_2(p)}{\mathcal{E}_0(p)-z}
\right),
\end{equation}
where $C\neq0$ is a normalization constant.
\end{theorem}

\begin{theorem}\label{theo:constant_2}
Let $\mu_2\in\mathbb{R}$, $\zeta\in\{0,1\}$ and
$\theta\in\{\mathrm{a}_{12},\mathrm{mix}\}$.
If $\mu_2\in\mathcal{B}_{\zeta}^{-}$ (resp.\ $\mu_2\in\mathcal{B}_{\zeta}^{+}$),
then $H^{\theta}_{\mu_2}(0)$ has exactly $\zeta$ eigenvalues lying strictly
below (resp.\ above) its essential spectrum.

Each such eigenvalue $z=z(\mu_2)$ admits an eigenfunction of the form
\begin{displaymath}
f_{\mu_2}^{\theta}(p)
= -\frac{\mu_2\,c\,\alpha_{3}^{\theta}(p)}
{\mathcal{E}_{0}(p)-z},
\end{displaymath}
where $c$ is a normalization constant.
\end{theorem}

\begin{remark}
We consider the operator $ H_{\mu_1\mu_2}(0) $, initially defined on 
$ L^{2,\mathrm{e}}(\mathbb{T}^3) $, as acting on the Banach space 
$ L^{1,\mathrm{e}}(\mathbb{T}^3) $ via
\begin{displaymath}
(H_{\mu_1\mu_2}(0)f)(p)
   = (H_{0}(0)f)(p) + (V_{\mu_1\mu_2}f)(p).
\end{displaymath}
The free part satisfies the estimate
\begin{displaymath}
\|H_0(0)f\|_{L^{1,\mathrm{e}}(\mathbb{T}^3)}
   \le 24\, \|f\|_{L^{1,\mathrm{e}}(\mathbb{T}^3)},
\end{displaymath}
while the interaction term obeys
\begin{displaymath}
\|V_{\mu_1\mu_2}f\|_{L^{1,\mathrm{e}}(\mathbb{T}^3)}
   \le \bigl(|\mu_1| + 6|\mu_2|\bigr)\,
      \|f\|_{L^{1,\mathrm{e}}(\mathbb{T}^3)}.
\end{displaymath}
Hence,
\begin{displaymath}
\|H_{\mu_1\mu_2}(0)f\|_{L^{1,\mathrm{e}}(\mathbb{T}^3)}
   \le \bigl(24 + |\mu_1| + 6|\mu_2|\bigr)\,
      \|f\|_{L^{1,\mathrm{e}}(\mathbb{T}^3)}.
\end{displaymath}
 Therefore, $ H_{\mu_1\mu_2}(0) $ is a bounded linear 
operator defined on $ L^{1,\mathrm{e}}(\mathbb{T}^3) $.
\end{remark}

\begin{definition}[Threshold resonance]
Let $(\mu_{1},\mu_{2})\in\mathbb{R}^{2}$.  
We say that the lower edge $\mathcal{E}_{\min}(0)$ of the essential spectrum is a 
\emph{threshold resonance} of the operator $H_{\mu_{1}\mu_{2}}(0)$ if the equation
\begin{displaymath}
H_{\mu_{1}\mu_{2}}(0) f(p) = \mathcal{E}_{\min}(0)\, f(p)
\end{displaymath}
admits a non-trivial solution
\begin{displaymath}
f \in L^{1,\mathrm{e}}(\mathbb{T}^{3}) 
\setminus L^{2,\mathrm{e}}(\mathbb{T}^{3}).
\end{displaymath}
Any such function $f$ is called a \emph{threshold resonance state}.
  
The notion of an \emph{upper threshold resonance} at $\mathcal{E}_{\max}(0)$ is defined analogously.
\end{definition}

\begin{remark}
A threshold resonance represents an intermediate spectral phenomenon between a
genuine eigenvalue and a regular point of the continuous spectrum.  
If the threshold equation admits a non-trivial solution in
$L^{1,\mathrm{e}}(\mathbb{T}^{3})$ that does not belong to
$L^{2,\mathrm{e}}(\mathbb{T}^{3})$, then the threshold $\mathcal{E}_{\min}(0)$
is "almost" an eigenvalue: the system is on the verge of creating a bound
state at the bottom of the essential spectrum.  
Such resonances act as precursors to the emergence (or disappearance) of
near-threshold eigenvalues under arbitrarily small parameter perturbations, and
they play an essential role in the analysis of Efimov-type effects and the
behavior of spectral counting functions near the threshold.
\end{remark}

The following theorem demonstrates that the number of eigenvalues located below the essential spectrum of the Hamiltonian $H^{\mathrm{s}} _{\mu_1\mu_2}(0)$ exhibits abrupt changes along the critical curves $\tau_0^-$ and $\tau_1^-$ and on these curves, the operator has a threshold resonance.
An analogous statement holds for eigenvalues above the essential spectrum with respect to $\tau^{+}_0$ and $\tau^{+}_1$.

\begin{theorem}\label{theo:resonance}
Let $(\mu_1^0,\mu_2^0)\in \tau_i^{-}, \,i=0,1$ and $n_{-}\!\bigl(H^{\mathrm{s}}_{\mu_1\mu_2}(0)\bigr)$ be the number of eigenvalues of $H^{\mathrm{s}}_{\mu_1\mu_2}(0)$ below the essential spectrum. Then
\begin{displaymath}
\lim_{\mathcal{A}_i^{-}\ni(\mu_1,\mu_2)\to(\mu_1^0,\mu_2^0)}
n_{-}\bigl(H^{\mathrm{s}}_{\mu_1\mu_2}(0)\bigr)=i, 
\qquad
\lim_{\mathcal{A}_{i+1}^{-}\ni(\mu_1,\mu_2)\to(\mu_1^0,\mu_2^0)}
n_{-}\bigl(H^{\mathrm{s}}_{\mu_1\mu_2}(0)\bigr)=i+1,
\end{displaymath}
and 
\begin{displaymath}
n_{-}\bigl(H^{\mathrm{s}}_{\mu_1^0\mu_2^0}(0)\bigr)=i.
\end{displaymath}

Moreover, if $(\mu_1^0,\mu_2^0)\in \tau^{-}_{i},$ $i=0,1$, 
then the bottom of the essential spectrum $\mathcal{E}_{\min}(0)$ is a 
threshold resonance of the operator $H^{\mathrm{s}}_{\mu_1^0\mu_2^0}(0)$.
\end{theorem}

\begin{remark}
An entirely analogous phenomenon occurs for the one--parameter operators 
$H^{\theta}_{\mu_2}(0)$, $\theta\in\{\mathrm{a_{12}},\mathrm{mix}\}$. 
At the critical values $\mu_2=\pm \mu_2^{(0)}$, their eigenvalue counting 
functions exhibit the same continuity properties and the same one-sided 
jumps at the thresholds $z=0$ and $z=24$ as in the symmetric case. 
Thus, the threshold spectral behavior of $H^{\theta}_{\mu_2}(0)$ is 
fully parallel to that of $H^{\mathrm{s}}_{\mu_1\mu_2}(0)$.
\end{remark}

We define
\begin{align*}
\mathbb{A}_{0}^{\mp}&=\mathcal{A}_{0}^{\mp}\cup \tau^{\mp}_0,\quad \mathbb{A}_{2}^{\mp}=\mathcal{A}_{2}^{\mp},\quad
\mathbb{A}_{3}^{\mp}=\mathcal{A}_{1}^{\mp}\cap\{(0,\mu_2): \mu_2\in\mathcal{B}_{1}^{\mp}\},\\
\mathbb{A}_{1}^{\mp}
&=\left(\mathcal{A}_{1}^{\mp}\cup \tau^{\mp}_1\right)\setminus\mathbb{A}_{3}^{\mp}.
\end{align*}
\begin{figure}[h!]
   \centering
    \begin{subfigure}[b]{0.45\textwidth}
        \centering
 \includegraphics[width=\textwidth]{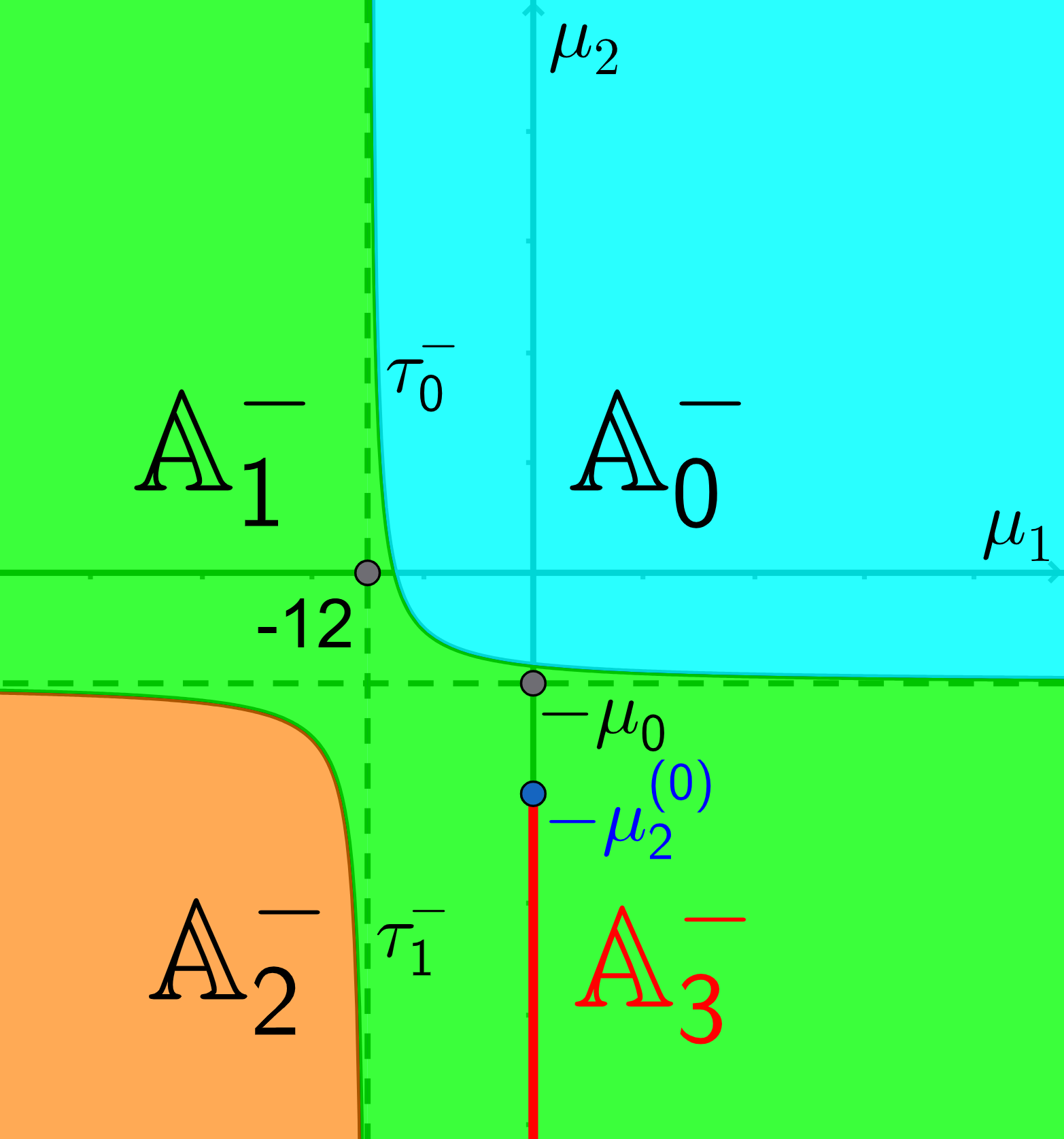}
       \caption{Partition of the $(\mu_1,\mu_2)$-plane by the curves $\tau_0^{-}$, $\tau_1^{-}$ and the point $(0,-\mu_2^{(0)})$.}
        \label{fig:rasm1}
    \end{subfigure}
    \hfill
    \begin{subfigure}[b]{0.45\textwidth}
       \centering
 \includegraphics[width=\textwidth]{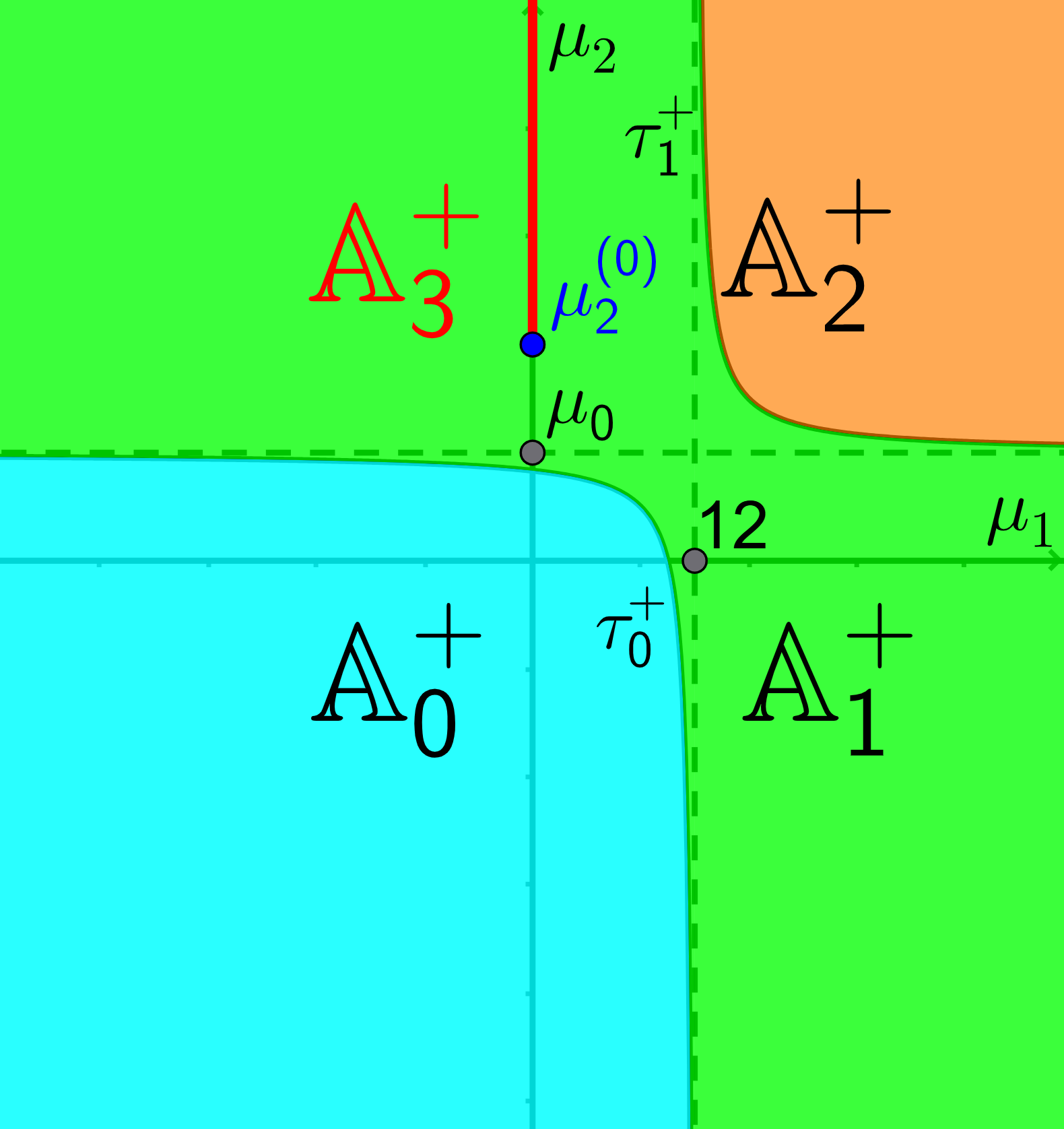}
       \caption{Partition of the $(\mu_1,\mu_2)$-plane by the curves $\tau_0^{+}$,$\tau_1^{+}$ and the point $(0,\mu_2^{(0)})$.}
        \label{fig:rasm2}
   \end{subfigure}
    \caption{}
   \label{fig:ikkala_rasm}
\end{figure} 

\begin{corollary}\label{theo:constant_3}
Let $\zeta\in\{0,1,2,3\}$.
The operator
$H_{\mu_1\mu_2}(0)$ has exactly $\zeta$ eigenvalues (counted with
multiplicity) lying below (resp.\ above) its essential spectrum if and only if $(\mu_1,\mu_2)\in\mathbb{A}_{\zeta}^{-}$
(resp.\ $(\mu_1,\mu_2)\in\mathbb{A}_{\zeta}^{+}$).
\end{corollary}

To determine the number of eigenvalues of the operator  
$H_{\mu_1 \mu_2}(0)$ lying on both sides of its essential spectrum,  
it is useful to analyze the geometric arrangement of the parameter regions  
$\mathcal{A}^{\mp}_{\alpha}$ for $\alpha \in \{0,1,2\}$ and  
$\mathcal{B}^{\mp}_{\beta}$ for $\beta \in \{0,1\}$.

\begin{lemma}\label{regions}
The following geometric relations hold for the regions 
$\mathcal{A}^{\mp}_{\alpha}$ and $\mathcal{B}^{\mp}_{\beta}$:

\begin{itemize}
\item[(i)] The inclusions
\begin{displaymath}
\mathcal{A}^{+}_2 \subset \mathcal{A}^{-}_0, 
\qquad 
\mathcal{A}^{-}_2 \subset \mathcal{A}^{+}_0,
\qquad 
\mathcal{B}^{+}_1 \subset \mathcal{B}^{-}_0, 
\qquad 
\mathcal{B}^{-}_1 \subset \mathcal{B}^{+}_0
\end{displaymath}
hold true.

\item[(ii)] The sets $\mathcal{A}^{-}_{\alpha}$ and $\mathcal{B}^{-}_{\beta}$ 
are symmetric with respect to the origin to the corresponding sets 
$\mathcal{A}^{+}_{\alpha}$ and $\mathcal{B}^{+}_{\beta}$:
\begin{displaymath}
\mathcal{A}^{-}_{\alpha} = -\mathcal{A}^{+}_{\alpha},
\qquad
\mathcal{B}^{-}_{\beta} = -\mathcal{B}^{+}_{\beta},
\end{displaymath}
where, for any set $S\subset\mathbb{R}^{2}$,
\begin{displaymath}
-S := \{\, -x : x\in S \,\}.
\end{displaymath}
\end{itemize}
\end{lemma}

\begin{proof}
(i)  
The inclusions follow directly from the definitions of the regions 
$\mathcal{A}^{\mp}_{\alpha}$ and $\mathcal{B}^{\mp}_{\beta}$ 
(see Lemmas \ref{lem:regions_A_pm} and \ref{lem:regions_B} 
), together with the monotonicity of the defining 
functions.  
For instance, if $(\mu_1,\mu_2)\in\mathcal{A}^{+}_{2}$, then
\begin{displaymath}
\mu_1>12, 
\qquad 
\mu_2>\frac{24}{\mu_1-12}+\mu_0.
\end{displaymath}
It follows immediately that
\begin{displaymath}
\mu_1>-12, 
\qquad 
\mu_2>\frac{24}{\mu_1+12}-\mu_0,
\end{displaymath}
and hence $(\mu_1,\mu_2)\in\mathcal{A}^{-}_{0}$.  
The remaining inclusions are verified analogously.

(ii)  
The symmetry follows from the definitions (see Lemma \ref{lem:regions_A_pm}):
\begin{displaymath}
(\mu_1,\mu_2)\in\mathcal{A}^{+}_{\alpha}
\quad\Longleftrightarrow\quad
(-\mu_1,-\mu_2)\in\mathcal{A}^{-}_{\alpha},
\end{displaymath}
\begin{displaymath}
\mu_2\in\mathcal{B}^{+}_{\beta}
\quad\Longleftrightarrow\quad
-\mu_2\in\mathcal{B}^{-}_{\beta}.
\end{displaymath}
This completes the proof.
\end{proof}

We introduce the intersection regions
\begin{displaymath}
\mathcal{G}_{\alpha\beta}
 := \mathbb{A}^{-}_{\alpha}\cap \mathbb{A}^{+}_{\beta},
 \qquad \alpha,\beta\in\{0,1,2,3\}.
\end{displaymath}

\begin{lemma}\label{regions_G}
The regions $\mathcal{G}_{\alpha\beta}$ satisfy the following properties:
\begin{itemize}

\item[(i)]
The intersections with indices $(2,0)$ and $(0,2)$ reduce to
\begin{displaymath}
\mathcal{G}_{20}
 = \mathbb{A}^{-}_2\cap\mathbb{A}^{+}_0
 = \mathbb{A}^{-}_2,
\qquad
\mathcal{G}_{02}
 = \mathbb{A}^{-}_0\cap\mathbb{A}^{+}_2
 = \mathbb{A}^{+}_2.
\end{displaymath}

\item[(ii)]
Similarly,
\begin{displaymath}
\mathcal{G}_{30}
 = \mathbb{A}^{-}_3\cap\mathbb{A}^{+}_0
 = \mathbb{A}^{-}_3,
\qquad
\mathcal{G}_{03}
 = \mathbb{A}^{-}_0\cap\mathbb{A}^{+}_3
 = \mathbb{A}^{+}_3.
\end{displaymath}

\item[(iii)]
If $\alpha+\beta>3$ or $(\alpha,\beta)\in\{(2,1),(1,2)\}$, then
\begin{displaymath}
\mathcal{G}_{\alpha\beta}=\varnothing.
\end{displaymath}

\item[(iv)]
If $\alpha+\beta<3$ or $(\alpha,\beta)\in\{(3,0),(0,3)\}$, then
\begin{displaymath}
(\mu_1,\mu_2)\in\mathcal{G}_{\alpha\beta}
\quad\Longleftrightarrow\quad
(-\mu_1,-\mu_2)\in\mathcal{G}_{\beta\alpha}.
\end{displaymath}
Thus the regions $\mathcal{G}_{\alpha\beta}$ and $\mathcal{G}_{\beta\alpha}$ 
are symmetric with respect to the origin in the $(\mu_1,\mu_2)$-plane.
\end{itemize}
\end{lemma}

\begin{proof}
(i)  
The identities follow directly from the definitions of 
$\mathbb{A}^{\mp}_{\alpha}$ together with Lemma \ref{regions},(i).

(ii)  
We verify the claim for $\mathcal{G}_{30}$; the argument for $\mathcal{G}_{03}$ 
is identical.  
By definition, any $(\mu_1,\mu_2)\in\mathbb{A}^{-}_3$ satisfies
\begin{displaymath}
\mu_1=0,
\qquad
\mu_2< -\mu_{2}^{(0)}.
\end{displaymath}
Since $\mu_{0}>2$, such points automatically lie in 
$\mathcal{A}^{+}_0\subset\mathbb{A}^{+}_0$, and therefore
\begin{displaymath}
\mathcal{G}_{30}
 = \mathbb{A}^{-}_3\cap\mathbb{A}^{+}_0
 = \mathbb{A}^{-}_3.
\end{displaymath}

(iii)  
Consider the case $(\alpha,\beta)=(1,2)$.  
Lemma \ref{regions},(i) implies 
$\mathcal{A}^{+}_2\subset\mathcal{A}^{-}_0$.  
Since 
$\mathcal{A}^{-}_0\cap \left(\mathcal{A}^{-}_1\cup\tau^{-}_{1}\right)=\varnothing$,
it follows that $\mathcal{G}_{12}=\varnothing$.  
The remaining cases are treated in the same way.

(iv)  
The proof follows directly from Lemma \ref{regions},(ii) and the symmetry properties of $\tau^{\mp}_{i},\,i=0,1.$
\end{proof}

\begin{corollary}\label{cor_G_ab}
Let $\mu_1,\mu_2\in\mathbb{R}$ and $\alpha,\beta\in\{0,1,2,3\}$ satisfy 
$\alpha+\beta<3$ or $(\alpha,\beta)\in\{(3,0),(0,3)\}$. The operator 
$H_{\mu_1\mu_2}(0)$ has exactly $\alpha$ eigenvalues below and 
$\beta$ eigenvalues above its essential spectrum if and only if $(\mu_1,\mu_2)\in\mathcal{G}_{\alpha\beta}$.
\end{corollary}
 
\noindent
\noindent
Corollary \ref{cor_G_ab} gives a complete classification of the discrete 
spectrum of $H_{\mu_1\mu_2}(0)$: for each parameter region 
$\mathcal{G}_{\alpha\beta}$, the operator has precisely $\alpha$ eigenvalues 
below and $\beta$ eigenvalues above its essential spectrum.  
This result provides a detailed description of how the location and number of 
discrete eigenvalues depend on the coupling constants $(\mu_1,\mu_2)$.  

In the next theorem, we extend this analysis to the family of lattice 
Schr\"{o}dinger operators $H_{\mu_1\mu_2}(K)$ with arbitrary total quasi-momentum 
$K\in\mathbb{T}^3$, establishing lower bounds on the number of eigenvalues 
lying below and above the essential spectrum for all $K$.

\begin{theorem}\label{teo:bound_eig}
Let $\mu_1, \mu_2 \in \mathbb{R}$ and $\alpha, \beta \in \{0,1,2,3\}$ satisfy  
$\alpha + \beta < 3$ or $(\alpha, \beta) \in \{(3,0),(0,3)\}$.  
If $(\mu_1, \mu_2) \in \mathcal{G}_{\alpha\beta}$, then for every 
$K \in \mathbb{T}^3$ the operator $H_{\mu_1\mu_2}(K)$ possesses 
\emph{at least} $\alpha$ eigenvalues below and \emph{at least} $\beta$ eigenvalues 
above its essential spectrum.
\end{theorem}

\section{Proofs of main results}\label{sec:Proofs}

\begin{proof}[Proof of Theorem \ref{theo:constant_1}]
By symmetry, it suffices to consider the ``$-$'' case.

The analyticity of the determinant
\begin{displaymath}
\Delta^{\mathrm{s}}_{\mu_1\mu_2}(z)
\end{displaymath}
in $z$ and $(\mu_1,\mu_2)$ implies that, by Rouch\'{e}'s theorem, the number of its zeros in $(-\infty,0)$ is locally constant.  
Since $\mathcal{A}=\mathcal{A}_{\zeta}^{-}$ is connected, this number is constant throughout $\mathcal{A}$ (see Theorem 3.2 in \cite{LMA:2023}).

It remains to find the number $n_{-}\!\big(H^{\mathrm{s}}_{\mu_1\mu_2}(0)\big)$ of eigenvalues of $H^{\mathrm{s}}_{\mu_1\mu_2}(0)$ below the essential spectrum in each region $\mathcal{A}_{\zeta}^{-}$.

(i) Case $\mathcal{A}_{0}^{-}$.  
By Lemma \ref{lem:regions_A_pm}, $(0,0)\in\mathcal{A}_{0}^{-}$.  
Since $H^{\mathrm{s}}_{\mu_1\mu_2}(0)=H_0(0)$ has no discrete spectrum below $\sigma_{\mathrm{ess}}(H^{\mathrm{s}}_{\mu_1\mu_2}(0))$, constancy gives
\begin{displaymath}
n_{-}\!\big(H^{\mathrm{s}}_{\mu_1\mu_2}(0)\big)=0.
\end{displaymath}

(ii) Case $\mathcal{A}_{1}^{-}$.
Since $(\mu_1,\mu_2)\in\mathcal{A}_{1}^{-}$ and $a^{\mathrm{s}}_{11}(0)>\frac{11}{102}$ (see Proposition \ref{lem:prop_function},(e)), which yields $\mu_0>2$, by the representation of $A^{-}(\mu_1,\mu_2)$ given by \eqref{polynomial_A}, we obtain $A^{-}(\mu_1,\mu_2)<0$. 
By Lemma \ref{lem:polynomial_A_pm},
$$
\lim_{z\to -\infty}\Delta^{\mathrm{s}}_{\mu_1\mu_2}(z)=1, \qquad 
\lim_{z\nearrow 0}\Delta^{\mathrm{s}}_{\mu_1\mu_2}(z)=A^{-}(\mu_1,\mu_2)<0,
$$
so $\Delta^{\mathrm{s}}_{\mu_1\mu_2}(z)$ has at least one zero in $(-\infty,0)$.  
Since $\operatorname{rank}(V_{\mu_1\mu_2}^{\mathrm{s}})\le2$, there are at most two zeros.  
The change of sign at the ends implies an odd number of zeros, hence: $n_{-}(H^{\mathrm{s}}_{\mu_1,\mu_2}(0)) = 1$.

(iii) Case $\mathcal{A}_{2}^{-}$.
Here $\mu_1<-12$ and $\mu_2<\frac{24}{\mu_1+12}-\mu_0<0$. As shown in (ii), we can also show that
\begin{equation}\label{A+}
A^{-}(\mu_1\mu_2)>0.
\end{equation} 
For $\mu_2=0$, by Lemma \ref{lem:polynomial_A_pm}
$$
\lim_{z\to -\infty}\Delta^{\mathrm{s}}_{\mu_1 0}(z)=1
$$
and since $\mu_1<-12$ and $a_{11}^{\mathrm{s}}(0)>\frac{11}{102}$,
$$ 
\lim_{z\nearrow 0}\Delta^{\mathrm{s}}_{\mu_1 0}(z)=1+\mu_1 a^{\mathrm{s}}_{11}(0)<0,
$$
so $\Delta^{\mathrm{s}}_{\mu_1 0}(z)$ has one zero $z_{11}<0$.  
For $\mu_1,\mu_2<0$,
\begin{displaymath}
\Delta^{\mathrm{s}}_{\mu_1\mu_2}(z_{11})
=-\mu_1\mu_2 a^{\mathrm{s}}_{12}(z_{11})^2<0,
\end{displaymath}
and by Lemma \ref{lem:polynomial_A_pm} and the inequality \eqref{A+},
\begin{displaymath}
\lim_{z\to -\infty}\Delta^{\mathrm{s}}_{\mu_1\mu_2}(z)=1, \qquad 
\lim_{z\nearrow 0}\Delta^{\mathrm{s}}_{\mu_1\mu_2}(z)=A^{-}(\mu_1,\mu_2)>0.
\end{displaymath}
Thus $\Delta^{\mathrm{s}}_{\mu_1\mu_2}(z)$ has two zeros $z_{21}<z_{11}<z_{22}<0$, and hence
\begin{displaymath}
n_{-}\!\big(H^{\mathrm{s}}_{\mu_1\mu_2}(0)\big)=2.
\end{displaymath}
Assume that $z:=z(\mu_{1},\mu_{2}) \in \mathbb{R}\setminus
\sigma_{\mathrm{ess}}(H_{\mu_{1}\mu_{2}}(0))$ is an eigenvalue of 
$H^{\mathrm{s}}_{\mu_{1}\mu_{2}}(0)$.
Then the eigenvalue equation
\begin{displaymath}
H^{\mathrm{s}}_{\mu_{1}\mu_{2}}(0)\, 
f^{\mathrm{s}}_{\mu_{1}\mu_{2}}
= z
\, f^{\mathrm{s}}_{\mu_{1}\mu_{2}}
\end{displaymath}
together with the existence of the resolvent 
$R_{0}(0,z):=(H_{0}(0)-zI)^{-1}$
implies
\begin{displaymath}
f^{\mathrm{s}}_{\mu_{1}\mu_{2}}
= -\, R_{0}(0,z)\,
V^{\mathrm{s}}_{\mu_{1}\mu_{2}}
\, f^{\mathrm{s}}_{\mu_{1}\mu_{2}}.
\end{displaymath}

Since the corresponding determinant satisfies
$\Delta^{\mathrm{s}}_{\mu_{1}\mu_{2}}(z)=0$,
the latter identity yields the explicit representation of the eigenfunction
$f^{\mathrm{s}}_{\mu_{1}\mu_{2}}$ given by \eqref{rep_eigenfunction}.
\end{proof}

\begin{proof}[Proof of Theorem \ref{theo:constant_2}] The proof is similar to the proof of Theorem \ref{theo:constant_1}.
\end{proof}
 \begin{proof}[Proof of Theorem \ref{theo:resonance}]
We prove the theorem for $i=0$, the case $i=1$ is analogous. 

By Theorem \ref{theo:constant_1}, the value of  
$n_{-}(H^{\mathrm{s}}_{\mu_1\mu_2}(0))$ is constant on each open region  
$\mathcal{A}^{-}_{\zeta}$ and equals $\zeta$.
If $(\mu_1^0,\mu_2^0)\in\tau_0^{-}$, then  
$\tau_0^{-}=\partial\mathcal{A}_0^{-}\cap\partial\mathcal{A}_1^{-}$.  
Therefore, the limits taken from $\mathcal{A}_0^{-}$ and $\mathcal{A}_1^{-}$ 
are $0$ and $1$, respectively. Since $n_{-}(H^{\mathrm{s}}_{\mu_1\mu_2}(0))$ is continuous as $(\mu_1,\mu_2)$ approaches any point on $\tau^{-}_0$ from within $\mathcal{A}^{-}_0$,  thus $n_{-}(H^{\mathrm{s}}_{\mu_1\mu_2}(0))=0$. 

We now show that if $(\mu_1^0,\mu_2^0)\in \tau^{-}_0$, then the equation  
\begin{displaymath}
H^{\mathrm{s}}_{\mu_1^0\mu_2^0}(0)f^{\mathrm{s}}_{\mu_1^0\mu_2^0}=0
\end{displaymath}
admits a non-trivial solution  
$f\in L^{1,\mathrm{e,s}}(\mathbb{T}^3)\setminus L^{2,\mathrm{e,s}}(\mathbb{T}^3)$; 
the other cases are handled analogously.

Let $(\mu_1^{0},\mu_2^{0})\in\tau^{-}_{0}$.  
By Lemma \ref{lem:polynomial_A_pm},(ii) and the definition of $\tau^{-}_{0}$  
(cf. Lemma \ref{lem:regions_A_pm}), we have
\begin{displaymath}
\Delta^{\mathrm{s}}_{\mu_1^0\mu_2^0}(0)
:=\lim_{z\to 0}\Delta^{\mathrm{s}}_{\mu_1^0\mu_2^0}(z)
=0.
\end{displaymath}
Since the Fredholm determinant extends continuously to $z=0$, 
Lemma \ref{lem:LS_and_det} also applies at $z=0$.  
Hence the equation $H^{\mathrm{s}}_{\mu_1^0\mu_2^0}(0)f^{\mathrm{s}}_{\mu_1^0\mu_2^0}=0$ has a non-trivial 
solution, obtained as the limit of the eigenfunction representation 
\eqref{rep_eigenfunction}.  
This gives
\begin{displaymath}
f^{\mathrm{s}}_{\mu_1^0\mu_2^0}(p)
=C\!\left(
\mu_1^{0}a_{12}^{\mathrm{s}}(0)\frac{\alpha_1^{\mathrm{s}}(p)}{\mathcal{E}_{0}(p)}
-
\mu_2^{0}\bigl(1+\mu_1^{0}a_{11}^{\mathrm{s}}(0)\bigr)\frac{\alpha^{\mathrm{s}}_{2}(p)}{\mathcal{E}_{0}(p)}
\right),
\qquad C\in\mathbb{R}\setminus\{0\}.
\end{displaymath}

Since $\alpha_{i}^{\mathrm{s}}(p)\sim 1$ and $\mathcal{E}_{0}(p)\sim|p|^{2}$ 
as $p\to 0$, we obtain $f^{\mathrm{s}}_{\mu_1^0\mu_2^0}\in L^{1,\mathrm{e,s}}(\mathbb{T}^3)$.  
On the other hand, because 
$(\alpha_{i}^{\mathrm{s}}(p))^{2}\sim 1$ whereas  
$\mathcal{E}_{0}(p)\sim |p|^{4}$, we also have  
$f^{\mathrm{s}}_{\mu_1^0\mu_2^0}\notin L^{2,\mathrm{e,s}}(\mathbb{T}^{3})$.  
Thus $\mathcal{E}_{\min}(0)=0$ is a threshold resonance of  
$H^{\mathrm{s}}_{\mu_1^0\mu_2^0}(0)$.
\end{proof}
\begin{proof}[Proof of Theorem \ref{teo:bound_eig}] Without loss of generality, we assume that $(\mu_1,\mu_2)\in\mathcal{G}_{30}$. By Corollary \ref{cor_G_ab} the fiber operator
$H_{\mu_1\mu_2}(0)$ has exactly three eigenvalues strictly below the bottom of the
essential spectrum. Denote these eigenvalues by
\begin{displaymath}
e_1(0;\mu_1,\mu_2)<e_2(0;\mu_1,\mu_2)<e_3(0;\mu_1,\mu_2)<\mathcal{E}_{\min}(0).
\end{displaymath}

For each quasi-momentum $K\in\mathbb{T}^3$ define the $m$-th variational level
\begin{displaymath}
e_m(K;\mu_1,\mu_2)
=\inf_{\substack{M\subset L^{2,\mathrm{e}}(\mathbb{T}^3)\\ \dim M = m}}
\;\sup_{\substack{\phi\in M\\ \|\phi\|=1}}
\bigl(H_{\mu_1\mu_2}(K)\phi,\phi\bigr),
\qquad m=1,2,3,
\end{displaymath}
so that, by the min--max principle (cf.\ \cite[Thm. XIII.1]{RSimon:IV}), $e_m(0;\mu_1,\mu_2)$
coincides with the $m$-th eigenvalue of $H_{\mu_1\mu_2}(0)$ counted from below.

By the monotonicity property of $ e_m(K;\mu_1,\mu_2)-\mathcal{E}_{\min}(K)$, as established in\\
Lemma 5.1 of \cite{LMA:2023}, for each fixed $m\in\{1,2,3\}$,
the quantity $e_m(K;\mu_1,\mu_2)-\mathcal{E}_{\min}(K)$ attains its maximum over
$\mathbb{T}^3$ at $K=0$. Therefore, for every $K\in\mathbb{T}^3$,
\begin{displaymath}
e_m(K;\mu_1,\mu_2)-\mathcal{E}_{\min}(K)
\le
e_m(0;\mu_1,\mu_2)-\mathcal{E}_{\min}(0)<0,
\end{displaymath}
which implies $e_3(K;\mu_1,\mu_2)<\mathcal{E}_{\min}(K)$.

Hence for each $K\in\mathbb{T}^3$ the operator $H_{\mu_1\mu_2}(K)$ has at least three
eigenvalues below its essential spectrum. This completes the proof for the case
$\mathcal{G}_{30}$.
\end{proof}

\section{Conclusions}
\label{sec:conclusions}

In this work, we performed a comprehensive spectral analysis of the two-particle lattice Hamiltonians $H_{\mu_1\mu_2}(0)$ with on-site and nearest-neighbor interactions. By decomposing the Hilbert space into three orthogonal invariant subspaces, we obtained a complete description of the regions in the $(\mu_1,\mu_2)$-plane and the intervals on the $\mu_2$-axis in which the number of eigenvalues lying below the essential spectrum remains constant, together with the analogous regions in the plane and intervals on $\mu_2$-axis for eigenvalues above the essential spectrum.

For the symmetric subspace, two smooth critical curves were identified. For the other two subspaces, a single critical point on the $\mu_2$-axis was found. 

These results provide a rigorous framework for understanding threshold phenomena, spectral transitions, and bound state formation in two-particle lattice systems, and they lay the groundwork for the study of more complex multi-particle interactions on discrete lattices.

\section*{Acknowledgments}
 The authors acknowledge the the support of this research by the Ministry of Innovative Development of the Republic of Uzbekistan (Grant No. FL-9524115052).
 
We used AI tools only for grammar and style editing.

\end{document}